\documentclass[10pt,leqno]{amsart}
\usepackage{graphicx}
\usepackage{indentfirst,csquotes}

\usepackage{amssymb,amsthm,amsmath,amsfonts,ascmac,mathrsfs,mathtools}
\usepackage{paralist,hyperref,fancyhdr,etoolbox}
\usepackage{tikz}
\usetikzlibrary{cd}

\theoremstyle{plain}
\newtheorem{theorem}{Theorem}[section]

\newtheorem{proposition}[theorem]{Proposition}
\newtheorem{corollary}[theorem]{Corollary}

\theoremstyle{remark}
\newtheorem{remark}[theorem]{Remark}
\theoremstyle{definition}
\newtheorem{definition}[theorem]{Definition}

\newcommand{\Z}{\mathbb{Z}}
\newcommand{\mr}{\mathrm}
\newcommand{\ra}{\rightarrow}

\hypersetup{
    colorlinks = true,
    linkcolor = [rgb]{0.3,0.55,0.25},
    citecolor = [rgb]{0.164,0.39,0.77},
    urlcolor = [rgb]{0.7,0.2,0.4}
}

\begin{document}
    \title{Homotopy Non-Invariance of the String Cobracket}
    \author{Isana SUMOTO}
    \address{Graduate School of Mathematical Sciences, The University of Tokyo, 3-8-1 Komaba, Meguro-ku, Tokyo, 153-8914, Japan}
    \email{isana(at)ms.u-tokyo.ac.jp}

    \begin{abstract}
        We prove that the integral string cobracket is not invariant under homotopy equivalences. Using Naef's computation \cite{Naef24} of the Goresky--Hingston coproduct and the Leray--Serre spectral sequence for the Borel construction, we compute the degree-two string cobracket for the three-dimensional lens spaces $L(9;1)$ and $L(9;4)$. The rank profiles of the cobracket on their nontrivial free-loop components differ, although the two lens spaces are homotopy equivalent. We also discuss the defect of naturality in relation to the Dennis trace of Whitehead torsion.
    \end{abstract} 
    \maketitle
    \begingroup
        \let\thefootnote\relax
    \endgroup
    \section{Introduction}
    String topology, introduced in \cite{ChasSullivan}, investigates algebraic structures on the homology of free loop spaces and was developed as a higher‑dimensional analogue of Goldman's work on closed curves on surfaces \cite{Goldman}. Among its fundamental operations are the string product and the string coproduct\footnote{While the string product is referred to as the ``loop product'' in \cite{ChasSullivan}, the terms ``loop coproduct'' and ``string coproduct'' often denote different operations. In this paper, following the convention in \cite{Naef24}, we refer to the product defined by Chas and Sullivan as the string product, and the coproduct investigated by Goresky and Hingston as the string coproduct.}. They induce operations on the $S^1$-equivariant homology of the free loop space, called the string bracket and the string cobracket. The string bracket generalizes the Goldman bracket, and the string cobracket generalizes the Turaev cobracket with rational coefficients \cite{HartensteinStegemeyer}.\\
    The string product is invariant under orientation-preserving homotopy equivalences \cite{Cohen08,Crabb08,GruherSalvatore08} (or could be deduced from \cite{CohenJones}). The behavior of the string coproduct is more subtle. For simply connected manifolds, the string coproduct is invariant under orientation-preserving homotopy equivalences over $\mathbb{Q}$ \cite{RiveraWang}. Moreover, it is invariant under orientation-preserving simple homotopy equivalences in general \cite{NaefSafronov}. However, Naef \cite{Naef24} computed the string coproduct on the three-dimensional lens spaces $L(7;1)$ and $L(7;2)$---which are homotopy equivalent but not homeomorphic---and obtained different values. Thus, the string coproduct is \emph{not} a homotopy invariant in general. Moreover, the string coproduct distinguishes non-homeomorphic three-dimensional lens spaces \cite{NaefManuelNathalie}. Naef and Safronov subsequently proved a general transformation formula describing the failure of homotopy invariance in terms of the Dennis trace of Whitehead torsion \cite{NaefSafronov}. An alternative formulation, based on parametrized fixed-point theory, was obtained by Kenigsberg and Porcelli \cite{KenigsbergPorcelli}; see \cite[Remark~1.5]{KenigsbergPorcelli} for its relation to the formula of Naef and Safronov.\\
    It is therefore natural to ask whether the failure of homotopy invariance of the string coproduct persists after passing to the $S^1$-equivariant theory, namely, for the string cobracket. In this paper, we focus on the \emph{string cobracket} of $L(9;1)$ and $L(9;4)$, a pair of manifolds that are homotopy equivalent but non-homeomorphic and for which explicit computations are feasible. The choice of $9$, rather than a prime such as $7$, is essential: the components indexed by elements of order $3$ have nontrivial low-degree $S^1$-equivariant homology, which makes it possible to distinguish the ranks of the two string cobrackets.\\
    The main contributions of this paper are as follows:
    \begin{enumerate}
        \item The integral string cobracket is \emph{not} invariant under orientation-preserving homotopy equivalences. More precisely, after reduction modulo $3$, the first K\"unneth components of the string cobrackets for $L(9;1)$ and $L(9;4)$ have ranks $8$ and $10$, respectively.
        \item The transformation formula of Naef and Safronov induces a transformation formula for the string cobracket. For a specific homotopy equivalence $L(9;1)\to L(9;4)$, we verify that the discrepancy between the two string cobrackets is accounted for by the Dennis trace of its Whitehead torsion.
    \end{enumerate}
    The paper is organized as follows. In Section \ref{section2}, we review the necessary background on string topology and the results of Naef \cite{Naef24}. In Section \ref{section3}, we compute the relevant $S^1$-equivariant homology groups and determine the string cobrackets of $L(9;1)$ and $L(9;4)$. In Section \ref{section4}, we relate the discrepancy between the two string cobrackets to the Dennis trace of Whitehead torsion.
    
    \section*{Acknowledgments}
    I would like to express my gratitude to my supervisor, Takuya Sakasai, for his invaluable advice and tremendous support. I also thank Kathryn Hess, Katsuhiko Kuribayashi, Takahito Naito, and Toyo Taniguchi for carefully reading the manuscript and providing constructive suggestions and insightful comments.\\\\
    \underline{\bf Use of AI.} The author used AI tools for proofreading, proof polishing, language polishing, formatting checks, and bibliographic metadata checks. The author assumes full responsibility for all mathematical content, arguments, citations, and conclusions in this manuscript.
    \section{Preparation}\label{section2}
    \subsection{Definitions}

    Throughout this paper, all homology and cohomology are taken with $\Z$-coefficients. Let $M$ be a smooth, closed, oriented $d$-dimensional manifold, and let $LM$ be its free loop space $\mathrm{Map}(S^1,M)$. We adopt the Sobolev $H^1$-loop space as our model for $LM$ (see \cite{HingstonWahl}); consequently, $LM$ is regarded as a Hilbert manifold.\\
    The \emph{string product}
    \begin{align*}
        \wedge\colon H_*(LM)\otimes H_*(LM)\ra H_{*-d}(LM)
    \end{align*}
    is an operation that equips the ($d$-shifted) homology of $LM$ with a ring structure. A second fundamental operation is the \emph{string coproduct}.
    \begin{align*}
        \vee\colon H_*(LM,M)\ra H_{*+1-d}(LM\times LM,LM\times M\cup M\times LM)
    \end{align*}
    is defined in \cite{Sullivan_2004} (also see \cite{GoreskyHingston}). Here, the subspace $M\subset LM$ is obtained by identifying each point $x\in M$ with the constant loop at $x$ via the natural inclusion $M\hookrightarrow LM$. The string bracket (resp. string cobracket) is defined by a combination of string product (resp. string coproduct) with the BV-operator introduced below.\\
    We first recall a geometric description of the string coproduct from \cite[Proposition 3.7]{HingstonWahl}. Let $N$ be an oriented compact $k$-dimensional manifold, and let $\alpha\colon N\ra LM$ be a smooth map whose adjoint $\widehat{\alpha}$ is smooth. Assume that:
    \begin{enumerate}
        \item For all $n\in N$, $\partial\alpha(n)(t)/\partial t$ is nonzero at $t=0$.
        \item The map $\bar{\alpha}\colon N\times (0,1)\ra M\times M$ given by $(n,t)\mapsto (\alpha(n)(0),\alpha(n)(t))$, $\bar{\alpha}$ is transverse to the diagonal in a compact submanifold of $N\times (0,1)$. Here, 
        \begin{align*}
            V\coloneqq \{(n,t)\in N\times S^1 \mid \alpha(n)(t)=\alpha(n)(0),t\neq 0\}
        \end{align*}
        is the self-intersection locus of $\alpha$.
    \end{enumerate}
    Let $x$ be the homology class given by $\alpha_*[N]\in H_*(LM,M)$. The map $\alpha$ induces
    \begin{align*}
        \vee(\bar{\alpha})\colon V\ra& LM\times LM\\
        (n,t)\mapsto& \left(s\mapsto \alpha(n)(st),s\mapsto \alpha(n)((1-t)s+t)\right).
    \end{align*}
    The string coproduct $\vee x$ coincides with
    \begin{align*}
        \vee x= \vee(\alpha_*[N])=(\vee\bar{\alpha})_*([V])\in H_{k+1-d}(LM\times LM,LM\times M\cup M\times LM).
    \end{align*}
    Second, we define the BV-operator, which plays a central role in the homology of free loop spaces.
    \begin{definition}[BV-operator]
        The free loop space $LM$ carries an $S^1$-action $\rho:S^1\times LM\ra LM;~(\theta,\gamma)\mapsto (t\mapsto \gamma(t+\theta))$. The \emph{Batalin--Vilkovisky (BV) operator} is the map
        \begin{align*}
            H_*(LM)&\ra H_{*+1}(LM)\\
            \alpha &\mapsto \rho_*([S^1]\otimes\alpha).
       \end{align*}
    \end{definition}
    For a space $X$ with an $S^1$-action, we denote its $S^1$-equivariant homology by $H_*^{S^1}(X)=H_*(ES^1\times_{S^1}X)$, where $X_{S^1} =ES^1 \times_{S^1} X$ is the Borel construction.\\
    Finally, following \cite{GoreskyHingston}, we will define the string cobracket on the $S^1$-equivariant homology of $LM$ as follows. Over $\Z$, the K\"unneth theorem does not in general identify the homology of the product pair with a tensor square, because $\mr{Tor}$-terms may occur. We therefore retain the product-pair target in the following definition.
    \begin{definition}[string cobracket]
        The string cobracket $\vee_{S^1}$ is defined by the composition  
        \begin{align*}
            \vee_{S^1}:&H_*^{S^1}(LM,M) \xrightarrow{\tau_*} H_{*+1}(LM,M)\\
            &\xrightarrow[\substack{\text{string}\\ \text{coproduct}}]{\vee} H_{*+2-d}(LM\times LM,LM\times M\cup M\times LM) \\
            &\xrightarrow{(\pi\times\pi)_*} H_{*+2-d}(LM_{S^1}\times LM_{S^1},LM_{S^1}\times M_{S^1}\cup M_{S^1}\times LM_{S^1})
        \end{align*}
        Here, $\tau$ is the relative $S^1$-equivariant transfer map. The inclusion $LM\hookrightarrow ES^1\times LM$ is a homotopy equivalence and since the diagram
        \[
            \begin{tikzcd}
                ES^1\times LM \arrow[d,"\simeq"]\arrow[rd,"\pi"]\\
                LM\arrow[r,"i"] &ES^1\times_{S^1} LM
            \end{tikzcd}
        \]
        commutes, we identify $i$ with $\pi$ by abuse of notation.
      \end{definition}
      In Section \ref{section3}, we study the maps $\tau_*$ and $\pi_*$ in more detail. The string bracket $\wedge_{S^1}$ is defined analogously.
      
      \subsection{Application of Naef's result}
      In the following, let $M$ be the lens space $L(9;k)$ with $k=1,4$. The lens spaces $L(9;1)$ and $L(9;4)$ are homotopy equivalent but not homeomorphic. Since $\pi_1(M)\cong \Z/9\Z$, the free loop space $LM$ decomposes into $9$ path-connected components, which we denote by $L_l M$ indexed by $l \in \Z/9\Z$.\\
      Following Naef \cite{Naef24}, the homology of $L_lM$ can be computed as follows:
    \begin{align*}
        H_0(L_lM)\cong \Z,\quad H_1(L_lM)\cong \Z/9\Z,\quad H_2(L_lM)\cong \Z,\\
        H_3(L_lM)\cong \Z\oplus(\Z/9\Z),\quad     H_4(L_lM)\cong \Z,~\dots.
    \end{align*}
    \begin{remark}
        In \cite{Naef24}, Naef computed the homology of the free loop space of the lens spaces $L(7;1)$ and $L(7;2)$. Since the argument does not rely on the fact that $7$ is prime, the same method applies to $L(9;k)$ without modification. On the other hand, in the latter part of our discussion, the fact that $9$ is not prime plays an essential role.
    \end{remark}
    \begin{proposition}\label{rho_0}
        Let $x_l \in H_1(L_lM)$ denote the positive generator. Then, the equality $\rho_*([S^1]\otimes 1_l)=l\cdot x_l$ holds. Here, $1_l$ is the canonical generator of $H_0(L_l M)$.
    \end{proposition}
    \begin{proof}
        Following \cite{Naef24}, we identify $H_0(LM)$ with $\Z[\Z/9\Z]$, and consider the module of formal de Rham 1-forms $\Omega^1=\Omega^1(\Z[\Z/9\Z])=((\Z/9\Z)[t]/(t^9-1))~dt/t$. The correspondence 
        \begin{align*}
            (\Z/9\Z)[t]/(t^9-1)~dt/t&\ra H_1(LM)=\bigoplus_{l\in\Z/9\Z} H_1(L_lM)
        \end{align*}
        given by
        \begin{align*}
            (c_0+c_1t+\dots+c_8t^8)\frac{dt}{t}&\mapsto (c_0,\dots,c_8)
        \end{align*}
        identifies coefficients with the generators of each $H_1(L_lM)$. Under this identification, the $S^1$ action induces the map $H_0(LM)\ra H_1(LM)$ corresponding to the formal de Rham differential
        \begin{align*}
            \Z[\Z/9\Z]&\ra (\Z/9\Z)[t]/(t^9-1)~dt/t\\
            (c_0+c_1t+\dots+c_8t^8)&\mapsto (c_1t+2c_2t^2+\dots+8c_8t^8)\frac{dt}{t}.
        \end{align*}
    \end{proof}
    Now, we consider the following component:
    \begin{align*}
        \vee_{S^1}\colon H_2^{S^1}(LM,M)\ra\left(H_1^{S^1}(LM,M)\otimes H_0^{S^1}(LM,M)\right)\oplus\left(H_0^{S^1}(LM,M)\otimes H_1^{S^1}(LM,M)\right).
    \end{align*}
    Since $\mr{Tor}\left(H_0^{S^1}(LM,M), H_0^{S^1}(LM,M)\right)=0$, there is no torsion term and the K{\"u}nneth decomposition is natural. Let $\mr{pr}_1$ denote the projection onto the first summand. Since $\mr{pr}_1$ commutes with $\pi_*\otimes \pi_*$, we may rearrange the order of composition as convenient.\\
    Let $t^pt_2^q~dt/t\in H_1(L_pM)\otimes H_0(L_qM)$ be the homology class whose projection to 
    \begin{align*}
        H_1(M)\otimes H_0(L_qM)\cong\Z/9\Z
    \end{align*}
    is the positive generator. For integers $l,m$, let $[\rho_{l,m}] \in H_3(L_lM)$ be the homology class obtained by pushing forward the fundamental class $[M]$ of $M$ along $\rho_{l,m}$. Here, $\rho_{l,m}\colon M\rightarrow LM$ is the map induced by
    \begin{align*}
        \tilde\rho_{l,m}\colon S^3&\ra LM\\
        (z_1,z_2)&\mapsto (t\mapsto p\left(e^{2\pi \sqrt{-1}lt/9}z_1,e^{2\pi \sqrt{-1}(kl+9m)t/9}z_2\right)
    \end{align*}
    where $p\colon S^3\ra M$ is the canonical quotient map. Using the identification $H_1(LM)\otimes H_0(LM)\cong R/I$, where $R=(\Omega^1/(\Z/9\Z)~dt/t)\otimes (\Z[\Z/9\Z]/\Z\cdot 1)$ and $I$ is the ideal generated by $t^9-1$ and $t_2^9-1$, the same argument as in \cite[Proposition 3.1]{Naef24} shows that for $l\neq 0$, the equality
    \begin{align}
        \begin{aligned}\label{string_coproduct_calc}
            \mr{pr}_1\circ\vee(j_*[\rho_{l,m}])&=\left(tt_2^{l-1}+t^2t_2^{l-2}+\cdots+t^{l-1}t_2\right)\frac{dt}{t}\\
            &+r\left(t^rt_2^{(kl+9m-1)r}+t^{2r}t_2^{(kl+9m-2)r}+\cdots+t^{(kl+9m-1)r}t_2^r\right)\frac{dt}{t} \mod{I}
        \end{aligned}
    \end{align}
    holds. Here, $j_*\colon H_*(LM)\ra H_*(LM,M)$ denotes the canonical homomorphism, and $r$ is the inverse of $k$ modulo $9$.
    \section{Computation of the string cobracket on the lens space}\label{section3}
    In this section, we compute the string cobracket on the lens space $L(9;k)$ for $k=1,4$. We calculate $H_r^{S^1}(L_lM)$ for $r=0,1,2,3$, then we investigate $\tau_*$ and $\pi_*$. We denote a nonzero element of $\Z/9\Z$ by $l$.

    \subsection{Computation of $H_r^{S^1}(L_lM)$ for $r\leq 3$}
    We analyze the Leray--Serre spectral sequence associated with the fibration
    \begin{align*}
        L_lM\ra ES^1\times_{S^1}L_lM\xrightarrow{\varpi} BS^1.
    \end{align*}
    \begin{proposition}
        Let $p$ be even. Then the differential of the $E^2$-term $d_{p,q}^2\colon E_{p,q}^2\ra E_{p-2,q+1}^2$ coincides with $u_p\otimes x\mapsto u_{p-2}\otimes \rho_*([S^1]\otimes x)$ where $u_p\in H_p(BS^1)$ denotes the generator.
    \end{proposition}
    \begin{proof}
        Let $p\geq 2$ be even and $(BS^1)^{(p)}$ denote the $p$-skeleton of $BS^1$ as $\mathbb{C}P^\infty$, and set $Y^p=\varpi^{-1}((BS^1)^{(p)})$. By the definition of the Leray--Serre spectral sequence, $E_{p,q}^1=H_{p+q}(Y^p,Y^{p-1})$. If $p$ is odd, then $Y^p$ is $Y^{p-1}$, and hence $E_{p,q}^1=0$. Therefore 
        \begin{align*}
            d_{p,q}^1\colon E_{p,q}^1\ra E_{p-1,q}^1
        \end{align*}
        is the zero map. Consequently, 
        \begin{align*}
            E_{p,q}^2=H_{p+q}(Y^p,Y^{p-1})=H_{p+q}(Y^p,Y^{p-2}).
        \end{align*}
        From now on, $p$ will be even.\\
        \underline{Step 1: The case $p=2$}\\
        We follow the proof of Theorem 5.3 in \cite{Hatcher}. Let $\Phi_2\colon D^2\ra (BS^1)^{(2)}$ be the characteristic map for the $2$-cell of $BS^1$. Let $e_0$ be a point in $\mr{int~}D^2$. Then we obtain a map $\tilde\Phi_2\colon (D^2,S^1)\times\varpi^{-1}(\Phi_2(e_0))\ra (Y^2,Y^1)$. As in \cite{Hatcher}, the map induces an isomorphism 
        \begin{align*}
            \varphi\colon H_2(S^{\Delta}((BS^1)^{(2)},(BS^1)^{(1)};\mathcal{H}_q))\ra& H_{q+2}(Y^2,Y^1)\\
            [\Phi_2]\otimes x \mapsto&(\tilde\Phi_2)_*(\mu\otimes x).
        \end{align*}
        where $\mu$ is the fundamental class of $(D^2,S^1)$, $S^{\Delta}(-,-;-)$ is the relative cellular singular chain complex and $\mathcal{H}_q$ is the local system on $BS^1$. Since $BS^1$ is simply-connected, the local system $\mathcal{H}_q$ is trivial. Thus,
        \begin{align*}
            H_2(S^{\Delta}((BS^1)^{(2)},(BS^1)^{(1)};\mathcal{H}_q))\cong H_2((BS^1)^{(2)},(BS^1)^{(1)})\otimes H_q(L_lM).
        \end{align*}
        The differential $d_{2,q}^2\colon H_{q+2}(Y^2,Y^0)\ra H_{q+1}(Y^0)$ is the boundary map in the long exact sequence of the pair $(Y^2,Y^0)$. Therefore,
        \begin{align*}
            d_{2,q}^2(\varphi ([\Phi_2]\otimes x))=\partial_*((\tilde\Phi_2)_*(\mu\otimes x)).
        \end{align*}
        By construction,
        \begin{align*}
            \tilde\Phi_2\colon (D^2,S^1)\times L_lM\ra& (Y_2,Y_0)
        \end{align*}
        is described as $(z,\gamma)\mapsto [[z:1],\gamma]$ and its restriction to $S^1\times L_lM$ is 
        \begin{align*}
            (\tilde\Phi_2)|_{S^1\times L_lM}\colon S^1\times L_lM\ra& Y_0\\
            (e^{2\pi \sqrt{-1}t},\gamma)\mapsto& [e^{2\pi \sqrt{-1}t},\gamma]=\rho(t,\gamma).
        \end{align*}
        Thus, $(\tilde\Phi_2)|_{S^1\times L_lM}=\rho$. By naturality of the boundary map, 
        $\partial_*((\tilde\Phi_2)_*(\mu\otimes x))=\rho_*(\partial_*(\mu\otimes x))=\rho_*([S^1]\otimes x)$ holds.\\
        
        \noindent\underline{Step 2: The case of general even $p$}\\
        Assume $p\geq 4$ and that the assertion holds for $p-2$. Let $e\in H^2(BS^1)$ be the Euler class of the tautological complex line bundle. Using the cap product and the naturality of the differential of the spectral sequence, we compute:
        \begin{align*}
            &(e\otimes 1)\cap (u_{p-2}\otimes \rho_*([S^1]\otimes x)-d_{p,q}^2(u_p\otimes x))\\
            &=u_{p-4}\otimes \rho_*([S^1]\otimes x)- d_{p,q}^2(e\otimes 1 \cap u_p\otimes x)\\
            &=u_{p-4}\otimes \rho_*([S^1]\otimes x)-d_{p-2,q}^2(u_{p-2}\otimes x)\\
            &=0.
        \end{align*}
        Since the cap product 
        \begin{align*}
            e\otimes 1\cap ~\colon H_p(BS^1)\otimes H_q(L_lM)\ra H_{p-2}(BS^1)\otimes H_q(L_lM)
        \end{align*}
        is an isomorphism, we conclude that $u_{p-2}\otimes \rho_*([S^1]\otimes x)=d_{p,q}^2(u_{p}\otimes x)$. This completes the proof.
    \end{proof}
    \begin{proposition}\label{rho_2}
        Let $y_l\in H_2(L_lM)$ be a positive generator. The equality $\rho_*([S^1]\otimes y_l)=9[\rho_{l,0}]+l\eta_{l,0}$ holds. Here, $\eta_{l,0}$ is a generator of the torsion summand of $H_3(L_lM)$.
    \end{proposition}
    \begin{proof}
        Let $p\colon S^3\ra M$ be the canonical quotient map, viewing $S^3$ as a subspace of $\mathbb{C}^2$. Let $(f_1,f_2)$ be an element in $\Omega S^3$. For a loop $(f_1,f_2)\in \Omega S^3$, the composition $p\circ (f_1,f_2)$ lies in $\Omega M$. Thus, we obtain a homeomorphism $\phi_l\colon\Omega S^3\ra\Omega_l M$ defined by
        \begin{align*}
            \phi_l (f_1,f_2)(t)= \rho_{l,0}([f_1(t),f_2(t)])(t).
        \end{align*}
        The map $\phi_l(f_1,f_2)$ represents $l$ times a generator of $\pi_1(M)$. Indeed, for $\alpha_l=\rho_{l,0}([1,0])$, since $(f_1,f_2)$ is contractible, a homotopy $H$ from the constant loop at $(1,0)$ to $(f_1,f_2)$ induces a homotopy  $\phi_l\circ H$ from $\alpha_l$ to $(t\mapsto p\circ (e^{2\pi\sqrt{-1}lt/9}f_1,e^{2\pi\sqrt{-1}klt/9}f_2))$.\\
        Let $[\mr{id}]\in\pi_3(S^3)$ denote the identity class. Under the correspondence $\pi_3(S^3)\cong \pi_2(\Omega S^3)$, we denote the representative of the class by $F$. We consider the homology of $L_lM$. Since $\Omega_lM$ is simply connected, the Hurewicz theorem yields $\pi_2(\Omega_l M)\cong H_2(\Omega_lM)$. Moreover, for the inclusion $\iota\colon \Omega_l M\ra L_lM$, the map $H_2(\Omega_l M)\xrightarrow{\iota_*}H_2(L_lM)$ is an isomorphism, so a generator of $H_2(L_lM)$ is $(\iota\circ\phi_l\circ F)_*[S^2]$. We consider the following commutative diagram
        \[
            \begin{tikzcd}
                \pi_3(S^3)\arrow[rd, phantom, "\circlearrowright"]\arrow[r,"\cong"']\arrow[d,"\cong"] & \pi_2(\Omega S^3)\arrow[r,"{\phi_l}","\cong"']\arrow[d,"\cong"]\arrow[rd, phantom, "\circlearrowright"] & \pi_2(\Omega_l M)\arrow[r,"\cong"']\arrow[d]\arrow[rd, phantom, "\circlearrowright"] & \pi_3(M)\arrow[d]\\
                H_3(S^3)\arrow[r] & H_2(\Omega S^3)\arrow[r]& H_2(\Omega_l(M))\arrow[r] & H_3(M)
            \end{tikzcd}.
        \]
        The vertical maps $\mr{H}$ are Hurewicz homeomorphisms, and $p_*\mr{H}([\mr{id}_{S^3}])=p_*[S^3]=9[M]$. Therefore,
        \begin{align*}
        ev_*\rho_*
        \left([S^1]\otimes(\iota\circ\phi_l\circ F)_*[S^2]\right)=9[M].
        \end{align*}
        Next, we compute the torsion term. Let $c\in H^2(L_l M)$ be an element such that $\iota^*c$ is a positive generator of $H^2(\Omega_l M)$. Set $\bar c=c-ev^*\rho_{l,0}^*c$ and choose $\eta_{l,0}$ satisfying $\bar c\cap \eta_{l,0}=x_l$. Since $H^1(L_lM)=0$, from the K{\"u}nneth formula, we obtain $H^2(S^1\times L_lM)\cong H^2(L_lM)$ and  $\rho^*\bar c=\mr{pr}_2^*\bar c$. Using projection formula, we obtain
        \begin{align*}
            \bar c\cap \rho_*([S^1]\times y_l)&=\rho_*(\rho^*\bar c\cap[S^1]\times y_l)\\
            &=\rho_*([S^1]\times (\bar c\cap y_l))\\
            &=\rho_*([S^1]\times 1_l)\\
            &= lx_l=\bar c\cap l\eta_{l,0}.
        \end{align*}
        Using the identification $H^*(L_lM)\cong H^*(M)\otimes H^*(\Omega_lM)$ by the Leray--Hirsch theorem, the map $\bar c\cap\colon H_3(L_lM)\ra H_1(L_lM)$ sends $n[\rho_{l,0}]+b\eta_{l,0}$ to $bx_l$. From $ev_*\rho_*([S^1]\times y_l)=9[M]$ and $\bar c\cap \rho_*([S^1]\times y_l)=lx_l$, we obtain $n=9$ and $b=l$. Hence, we obtain $\rho_*([S^1]\otimes y_l)=9[\rho_{l,0}]+l\eta_{l,0}$.
    \end{proof}

    Combining these two propositions and Proposition \ref{rho_0}, we compute $H_r^{S^1}(L_lM)$ for $r=0,1,2,3$. We see that the low-degree $E^2$-page of the Leray-Serre spectral sequence for the Borel fibration of $L_lM$ under the natural identification $E_{p,q}^2=H_p(BS^1)\otimes H_q(L_lM)$ is:
    \begin{align*}
        \begin{array}{c|ccccc}
            3 & \Z\oplus\Z/9\Z & 0 & \Z\oplus\Z/9\Z & 0 & \Z\oplus\Z/9\Z\\
            2 & \Z & 0 & \Z & 0 & \Z\\
            1 & \Z/9\Z & 0 & \Z/9\Z & 0 & \Z/9\Z\\
            0 & \Z & 0 & \Z & 0 & \Z\\
            \hline
            & 0 & 1 & 2 & 3 & 4
            \end{array}
    \end{align*}
    Because $L_lM$ is a path-connected component, $H_0^{S^1}(L_lM)\cong \Z$. Moreover, Proposition \ref{rho_0} gives $\mr{Im}d_{2,0}^2\cong l~\Z/9\Z$, hence $H_1^{S^1}(L_lM)=0$ for $l\neq 3,6$ and $H_1^{S^1}(L_lM)\cong \Z/3\Z$ for $l=3,6$\footnotemark. $H_2^{S^1}(L_lM)$ is isomorphic to $\Z^2$. Proposition \ref{rho_2} implies that $E_{0,3}^3\cong \Z/81\Z$ for $l\neq 3,6$ and $\Z/3\Z\oplus \Z/27\Z$ for $l=3,6$. Hence $E_{1,2}^\infty=E_{3,0}^\infty=0$ holds and $E_{2,1}^3=0$ for $l\neq 3,6$ and $E_{2,1}^3\cong \Z/3\Z$ for $l=3,6$.\footnotetext{This phenomenon arises from the fact that $9$ is not prime, which allows nontrivial $S^1$-equivariant homology in the components $L_3M$and $L_6M$, and it does not occur for prime lens spaces such as $L(7;k)$.} Therefore, $H_3^{S^1}(L_lM)$ satisfies the exact sequence
    \begin{align*}
        0\ra E_{0,3}^5 \ra H_3^{S^1}(L_lM)\ra E_{2,1}^3\ra 0.
    \end{align*}
    Finally, we prove the following proposition. Using this proposition, $E_{0,3}^5=0$ for $l\neq 3,6$ and $E_{0,3}^5\cong (\Z/3\Z)^2$ for $l=3,6$.
    
    \begin{proposition}
        For the positive generator of $H_4(BS^1)$,
        \begin{align*}
            d_{4,0}^4\left(\frac{9}{\gcd(l,9)}u_4\right)=\frac{kl^2}{\gcd(l,9)}\bar{a}
        \end{align*}
        holds. Here, $\bar{a}$ is the image of $[\rho_{l,0}]$.
    \end{proposition}
    \begin{proof}
        Recall the following $S^1$-actions on $S^3$ and $M$ defined by
        \begin{align*}
            S^1\times S^3&\ra S^3\\
            (s,z_1,z_2)&\mapsto (s^l z_1,s^{kl+9m}z_2),\\
            S^1\times M&\ra M\\
            (s,[z_1,z_2])&\mapsto \rho_{l,m}([z_1,z_2])(s).
        \end{align*}
        Using $\Phi_9\colon S^1\ra S^1; s\mapsto s^9$, for $z\in S^3$, we see that
        \begin{align*}
            p(s\cdot (z_1,z_2))=\Phi_9(s)\cdot p(z_1,z_2).
        \end{align*}
        Hence, we obtain the following map between the fibrations
        \[
            \begin{tikzcd}
                S^3\arrow[r]\arrow[d,"p"]&S^3_{S^1}\arrow[r]\arrow[d]&BS^1\arrow[d,"B\Phi_9"]\\
                M\arrow[r] &M_{S^1}\arrow[r]&BS^1
            \end{tikzcd}.
        \]
        Let $\mathbb C_r$ denote the one-dimensional complex representation of $S^1$ of weight $r$, and let
        \begin{align*}
            L_r\coloneqq ES^1\times_{S^1}\mathbb C_r\longrightarrow BS^1
        \end{align*}
        be the associated complex line bundle. Since \(S^3=S(V)\) is the unit sphere in $V=\mathbb C_l\oplus\mathbb C_{kl+9m}$, its Borel construction is the sphere bundle $S(L^{\otimes l}\oplus L^{\otimes(kl+9m)})$. Let $L$ be the universal line bundle $ES^1\times_{S^1} \mathbb{C}_1$. Since the associated line bundle $L_r\cong L^{\otimes r}$,
        we see that the bundle $S^3_{S^1}$ is the sphere bundle of $L^{\otimes l}\oplus L^{\otimes (kl+9m)}$. Hence, the Euler class $e_{S^3}$ of $S^3_{S^1}$ is
        \begin{align*}
            e_{S^3}=e(L^{\otimes l}\oplus L^{\otimes (kl+9m)})=c_1(L^{\otimes l})c_1(L^{\otimes kl+9m})=l(kl+9m)c_1(L)^2.
        \end{align*}
        Therefore, the Gysin sequence of this $S^3$-bundle shows that the transgression in the Leray--Serre spectral sequence is 
        \begin{align*}
            d_4^{S^3}(u_4)=\langle e_{S^3},u_4\rangle [S^3]=l(kl+9m)[S^3]
        \end{align*}
        Indeed, $B\Phi_9^*c_1(L)=9c_1(L)$, and hence $B\Phi_9^*c_1(L)^2=81c_1(L)^2$. Since $p_*[S^3]=9[M]$, using naturality, the transgression $d_4^M$ of the Leray--Serre spectral sequence for the fibration $M\ra M_{S^1}\ra BS^1$ is $d_4^{M}(81u_4)=9l(kl+9m)[M]$. Since $E_{4,0}^4=9/\gcd{(l,9)}\Z u_4$ holds, we see
        \begin{align*}
                d_4^M\left(v_4\right)=\frac{l(kl+9m)}{\gcd{(l,9)}}[M].
        \end{align*}
        Here, 
        \begin{align*}
            v_4\coloneqq \frac{9}{\gcd{(l,9)}}u_4.
        \end{align*}
        Since $\rho_{l,m}$ is $S^1$ equivariant, by naturality, we obtain 
        \begin{align*}
            d_{4,0}^4(v_4)=d_4(v_4)&={\rho_{l,m}}_*d_4^M(v_4)\\
            &=\frac{l(kl+9m)}{\gcd{(l,9)}}[\rho_{l,m}].
        \end{align*}
        Substituting $m=0$, we obtain the desired formula.
    \end{proof}

    \subsection{Computation of $\tau_*$}
    We determine the image of $\tau_*$ by considering the Gysin exact sequence for $S^1$-bundle $ES^1\times L_lM\ra ES^1\times_{S^1} L_lM$. First, suppose that $l\notin\{0,3,6\}$. Since the diagram
    \[
    \begin{tikzcd}
        H_2^{S^1}(L_lM)\arrow[r,"\tau_*"]\arrow[d,"\cong"]\arrow[rd, phantom, "\circlearrowright"]& H_3(L_lM) \arrow[r,"\pi_*"]\arrow[d,"\cong"]\arrow[rd, phantom, "\circlearrowright"]& H_3^{S^1}(L_lM) \arrow[d,"\cong"]\\
        \Z^2\arrow[r]& \Z\oplus(\Z/9\Z)\arrow[r]& 0
    \end{tikzcd}
    \]
    commutes and the row is exact, $\tau_*$ is surjective.\\
    For $l=3,6$, the image of $\pi_*$ coincides with $H_3(L_lM)\twoheadrightarrow E_{0,3}^\infty\hookrightarrow H_3^{S^1}(L_lM)$ (cf. \cite[Theorem 5.9]{Hatcher}). Hence, $\mr{Im}\pi_*\cong (\Z/3\Z)^2$ holds. Thus, we obtain
    \[
        \begin{tikzcd}
            H_2^{S^1}(L_lM)\arrow[r,"\tau_*"]\arrow[d,"\cong"]\arrow[rd, phantom, "\circlearrowright"]& H_3(L_lM) \arrow[r,->>,"\pi_*"]\arrow[d,"\cong"]\arrow[rd, phantom, "\circlearrowright"]& E_{0,3}^5\arrow[r] \arrow[d,"\cong"]\arrow[rd, phantom, "\circlearrowright"]&0\arrow[d,"\cong"]\\
            \Z^2\arrow[r]& \Z\oplus(\Z/9\Z)\arrow[r,->>]& (\Z/3\Z)^2\arrow[r]&0
        \end{tikzcd}
        .
    \]
    It follows from the exactness of the Gysin sequence and
    \begin{align*}
        \mr{Ker} \pi_*&=\mr{Im}d_{2,2}^2+\mr{Im}d_{4,0}^4\\
        &=\left\langle 9[\rho_{l,0}]+l\eta_{l,0}, \frac{kl^2}3 [\rho_{l,0}] \right\rangle
    \end{align*}
    that
    \begin{align*}
        \mr{Im}\tau_*=
    \begin{cases}
        H_3(L_lM), & l\notin\{3,6\},\\
        3H_3(L_lM), & l\in\{3,6\}.
    \end{cases}
    \end{align*}
    Consequently, for $l\notin\{3,6\}$, we may choose a basis $z_l,w_l\in H_2^{S^1}(L_lM)$ such that
    \begin{align*}
        \tau_*z_l=[\rho_{l,0}],\quad \tau_*w_l=\eta_{l,0}.
    \end{align*}
    For $l\in\{3,6\}$, we may choose a basis $z_l,w_l\in H_2^{S^1}(L_lM)$ such that
    \begin{align*}
        \tau_*z_l=3[\rho_{l,0}],\quad\tau_*w_l=3\eta_{l,0}.
    \end{align*}
    
    \subsection{Computation of $\pi_*$}\label{pi_*}
    For the $S^1$-bundle $\pi\colon ES^1\times L_lM\rightarrow ES^1\times_{S^1}L_lM$, we compute $\pi_*\colon H_r(L_lM)\ra H_r^{S^1}(L_lM)$ for $r=0,1$.\\
    For $r=0$, the map $\pi_*$ is an isomorphism by the Gysin exact sequence of the $S^1$-bundle $ES^1\times L_lM\ra ES^1\times_{S^1} L_lM$.\\
    Next, we consider $r=1$. If $l\neq 3,6$, then $H_1^{S^1}(L_lM)=0$. By contrast, for $l=3,6$, the Gysin exact sequence of the $S^1$-bundle $ES^1\times L_lM\ra ES^1\times_{S^1} L_lM$ near degrees $0,1$ yields:
    \[
        \begin{tikzcd}
            H_1(L_lM)\arrow[r,->>,"\pi_*"]\arrow[d,"\cong"]\arrow[rd, phantom, "\circlearrowright"]& H_1^{S^1}(L_lM) \arrow[r]\arrow[d,"\cong"]\arrow[rd, phantom, "\circlearrowright"]& 0\arrow[d,"\cong"]\\
            \Z/9\Z\arrow[r,->>]& \Z/3\Z\arrow[r]&0
        \end{tikzcd}
    \]
    and 
    \[
        \begin{tikzcd}
            H_0(L_lM)\arrow[r,->>,"\pi_*"]\arrow[d,"\cong"]\arrow[rd, phantom, "\circlearrowright"]& H_0^{S^1}(L_lM) \arrow[r]\arrow[d,"\cong"]\arrow[rd, phantom, "\circlearrowright"]& 0\arrow[d,"\cong"]\\
            \Z\arrow[r,->>]& \Z\arrow[r]&0
        \end{tikzcd}
        .
    \]
    From these commutative diagrams, the image of $ct^pt_2^qdt/t~(c\in\Z/9\Z)$ by $(\pi\times\pi)_*$ is $0$ for $p\neq 3,6$ and $c'\alpha_p\otimes \beta_q$ for $p=3,6$. Here, $c'$ is the reduction of $c$ modulo $3$, while $\alpha_p$ and $\beta_q$ are generators of $H_1^{S^1}(L_pM)$ and $H_0^{S^1}(L_qM)$, respectively.
    \subsection{Computation of the string cobracket on $L(9;1)$ and $L(9;4)$}
    Using the formula (\ref{string_coproduct_calc}), we compute $\mr{pr}_1\circ\vee(j_*[\rho_{l,0}])$ for $l\neq 0$. The resulting expressions, organized by $l$, are listed in Table \ref{string_coproduct_1} for $k=1$ and Table \ref{string_coproduct_4} for $k=4$.
    \begin{table}[htbp]
        \centering
        \begin{align*}
            \begin{array}{c|l}
                l&\mr{pr}_1\circ\vee(j_*[\rho_{l,0}])\\
                \hline
                1&0\\
                2&2t^{1}t_2^{1}\frac{dt}{t}\\
                3&2t^{1}t_2^{2}+2t^{2}t_2^{1}\frac{dt}{t}\\
                4&2t^{1}t_2^{3}+2t^{2}t_2^{2}+2t^{3}t_2^{1}\frac{dt}{t}\\
                5&2t^{1}t_2^{4}+2t^{2}t_2^{3}+2t^{3}t_2^{2}+2t^{4}t_2^{1}\frac{dt}{t}\\
                6&2t^{1}t_2^{5}+2t^{2}t_2^{4}+2t^{3}t_2^{3}+2t^{4}t_2^{2}+2t^{5}t_2^{1}\frac{dt}{t}\\
                7&2t^{1}t_2^{6}+2t^{2}t_2^{5}+2t^{3}t_2^{4}+2t^{4}t_2^{3}+2t^{5}t_2^{2}+2t^{6}t_2^{1}\frac{dt}{t}\\
                8&2t^{1}t_2^{7}+2t^{2}t_2^{6}+2t^{3}t_2^{5}+2t^{4}t_2^{4}+2t^{5}t_2^{3}+2t^{6}t_2^{2}+2t^{7}t_2^{1}\frac{dt}{t}
            \end{array}
        \end{align*}
        \caption{String coproduct for $k=1$}
        \label{string_coproduct_1}
    \end{table}
    \begin{table}[htbp]
        \centering
        
        \begin{align*}
            \begin{array}{c|l}
                l&\mr{pr}_1\circ\vee(j_*[\rho_{l,0}])\\
                \hline
                1&7t^{3}t_2^{7}+7t^{5}t_2^{5}+7t^{7}t_2^{3}\frac{dt}{t}\\
                2&8t^{1}t_2^{1}+7t^{3}t_2^{8}+7t^{4}t_2^{7}+7t^{5}t_2^{6}+7t^{6}t_2^{5}+7t^{7}t_2^{4}+7t^{8}t_2^{3}\frac{dt}{t}\\
                3&8t^{1}t_2^{2}+8t^{2}t_2^{1}+7t^{4}t_2^{8}+5t^{5}t_2^{7}+7t^{6}t_2^{6}+5t^{7}t_2^{5}+7t^{8}t_2^{4}\frac{dt}{t}\\
                4&6t^{1}t_2^{3}+8t^{2}t_2^{2}+6t^{3}t_2^{1}+5t^{5}t_2^{8}+5t^{6}t_2^{7}+5t^{7}t_2^{6}+5t^{8}t_2^{5}\frac{dt}{t}\\
                5&6t^{1}t_2^{4}+6t^{2}t_2^{3}+6t^{3}t_2^{2}+6t^{4}t_2^{1}+5t^{6}t_2^{8}+3t^{7}t_2^{7}+5t^{8}t_2^{6}\frac{dt}{t}\\
                6&4t^{1}t_2^{5}+6t^{2}t_2^{4}+4t^{3}t_2^{3}+6t^{4}t_2^{2}+4t^{5}t_2^{1}+3t^{7}t_2^{8}+3t^{8}t_2^{7}\frac{dt}{t}\\
                7&4t^{1}t_2^{6}+4t^{2}t_2^{5}+4t^{3}t_2^{4}+4t^{4}t_2^{3}+4t^{5}t_2^{2}+4t^{6}t_2^{1}+3t^{8}t_2^{8}\frac{dt}{t}\\
                8&2t^{1}t_2^{7}+4t^{2}t_2^{6}+2t^{3}t_2^{5}+4t^{4}t_2^{4}+2t^{5}t_2^{3}+4t^{6}t_2^{2}+2t^{7}t_2^{1}\frac{dt}{t}
            \end{array}
        \end{align*}
        \caption{String coproduct for $k=4$}
        \label{string_coproduct_4}
    \end{table}\\
    Note that the string cobracket on $L_0M$ is the zero map since
    \begin{align*}
        \tau_*j_*\pi_*y_0=j_*\rho_*([S^1]\otimes y_0)=9j_*[\rho_{0,0}]=0
    \end{align*}
    holds. We also remark that the following diagram commutes
    \[
        \begin{tikzcd}
            H_2^{S^1}(LM)\arrow[r,"j_*"]\arrow[d,"\tau_*"]\arrow[rd, phantom, "\circlearrowright"]&H_2^{S^1}(LM,M)\arrow[d,"\tau_*"]\\
            H_3(LM)\arrow[r,"j_*"] &H_3(LM,M).
        \end{tikzcd}
    \]
    From the long exact sequences of the pair
    \[
        \begin{tikzcd}
            H_2^{S^1}(M)\arrow[r]\arrow[d,"\cong"]\arrow[rd, phantom, "\circlearrowright"]& H_2^{S^1}(L_0M)\arrow[r]\arrow[d,"\cong"]& H_2^{S^1}(L_0M,M)\arrow[r]& H_1^{S^1}(M)\arrow[r,"\cong"]\arrow[d,"\cong"]\arrow[rd, phantom, "\circlearrowright"]&H_1^{S^1}(L_0M)\arrow[d,"\cong"]\\
            \Z \arrow[r]& \Z^2&& \Z/9\Z\arrow[r,"\cong"]&\Z/9\Z,
        \end{tikzcd}
    \]
    the generators of $H_2^{S^1}(LM,M)$ are $j_*z_l$ and $j_*w_l$ for $l\neq 0$ and $j_*\pi_*y_0$. Therefore, we consider only $\vee_{S^1}(j_*z_l)$ and $\vee_{S^1}(j_*w_l)$. 
    
    We now compute $\mr{pr}_1\circ\vee_{S^1}(j_*w_l)$. For $l\neq 3,6$,
    \begin{align*}
        \mr{pr}_1\circ\vee_{S^1}(j_*w_l)=rd_l(\alpha_3\otimes\beta_{l-3}+\alpha_6\otimes\beta_{l-6})
    \end{align*}
    and for $l=3,6$, the image is zero. Here, $d_l$ satisfies $\eta_{l,0}=[\rho_{l,d_l}]-[\rho_{l,0}]$. All indices are taken modulo $9$. Since 
    \begin{align*}
        \bar c\cap ([\rho_{l,d_l}]-[\rho_{l,0}])=rd_lx_l,
    \end{align*}
    setting $d_l=k$, we conclude the following:
    \begin{align}
        \mr{pr}_1\circ\vee_{S^1}(j_*w_l)=
        \begin{cases}
            \alpha_3\otimes\beta_{l-3}+\alpha_6\otimes\beta_{l-6},&(l\neq 3,6),\\
            0, &(l=3,6).\label{string_cobracket_K}
        \end{cases}
    \end{align}
    For $k=1,4$, let $\delta_k \coloneqq \mr{pr}_1 \circ \vee_{S^1}$ denote the first K{\"u}nneth component of the string cobracket on $L(9;k)$, reduced modulo $3$. We also set
    \begin{align}
        W_k\coloneqq \mr{span}_{\Z/3\Z}\{\delta_k(j_*w_l)\mid l=1,\dots,8\}.
    \end{align}
    Using Proposition \ref{pi_*} and Tables \ref{string_coproduct_1} and \ref{string_coproduct_4}, we compute $\delta_k(j_*z_l)$ modulo $W_k$. The results are shown in Tables \ref{string_cobracket_1} and \ref{string_cobracket_4}.
    \begin{table}[htbp]
        \centering
        \begin{minipage}[t]{0.48\textwidth}
            \begin{align*}
                \begin{array}{c|l}
                    l & \delta_1(j_*z_l)\mod{W_1}\\
                    \hline
                    1 & 0\\
                    2 & 0\\
                    3 & 0\\
                    4 & 2\alpha_3\otimes\beta_1\\
                    5 & 2\alpha_3\otimes\beta_2\\
                    6 & 0\\
                    7 & 0\\
                    8 & 0
                \end{array}
            \end{align*}
            \caption{String cobracket for $k=1$}
            \label{string_cobracket_1}
            \centering
        \end{minipage}
        \begin{minipage}[t]{0.48\textwidth}
            \begin{align*}
                \begin{array}{c|l}
                    l & \delta_4(j_*z_l)\mod{W_4}\\
                    \hline
                    1 & \alpha_3\otimes\beta_7\\
                    2 & 0\\
                    3 & 0\\
                    4 & 2\alpha_6\otimes\beta_7\\
                    5 & 2\alpha_6\otimes\beta_8\\
                    6 & 0\\
                    7 & 0\\
                    8 & \alpha_3\otimes\beta_5
                \end{array}
            \end{align*}
            \caption{String cobracket for $k=4$}
            \label{string_cobracket_4}
        \end{minipage}
    \end{table}\\
    Consequently, one can see that the string cobrackets for $k=1,4$ differ. Using this calculation, we obtain the following theorem.
    \begin{theorem}
        The string cobracket is not a homotopy invariant.
    \end{theorem}
    \begin{proof}
        Equation \eqref{string_cobracket_K} shows that $\dim_{\Z/3\Z} W_k = 6$. After passing to the quotient by $W_k$, Tables \ref{string_cobracket_1} and \ref{string_cobracket_4} show that the ranks of the induced maps are $2$ for $k=1$ and $4$ for
        $k=4$, respectively. Consequently,
        \begin{align*}
            \mr{rank}_{\Z/3\Z} \delta_1 = 6+2=8,
            \qquad
            \mr{rank}_{\Z/3\Z} \delta_4 = 6+4=10.
        \end{align*}
        Suppose that an orientation-preserving homotopy equivalence $f\colon L(9;1) \ra L(9;4) $ intertwined the string cobrackets. Since the K{\"u}nneth projection $\mr{pr}_1$ is natural with respect to $f$, the induced isomorphisms would also intertwine $\delta_1$ and $\delta_4$. In particular, these two maps would have the same rank over $\Z/3\Z$, contradicting $8 \neq 10$. Therefore, no homotopy equivalence between $L(9;1)$ and $L(9;4)$ intertwines the string cobrackets.
    \end{proof}
    
    \section{Relationship between the string cobracket and Whitehead torsion} \label{section4}
    We now investigate how the string cobracket transforms under homotopy equivalence and how this behavior reflects Whitehead torsion. Let
    \begin{align*}
        M_k\coloneqq L(9;k) \quad (k=1,4)
    \end{align*}
    and fix an orientation-preserving homotopy equivalence $f\colon M_1\ra M_4$ that sends the component $L_lM_1$ to $L_{4l}M_4$. We first recall the transformation formula of Naef and Safronov \cite{NaefSafronov}. For a closed oriented $d$-manifold $P$, consider the two relative antidiagonal maps
    \begin{align*}
        \Delta_L\colon (LP,P)&\ra (LP\times LP,P\times LP),\\
        \Delta_R\colon (LP,P)&\ra (LP\times LP,LP\times P),
    \end{align*}
    both of which send $\gamma$ to $(\gamma^{-1},\gamma)$.
    Let
    \begin{align*}
        q_L\colon& H_*(LP\times LP,P\times LP) \ra H_*(LP\times LP,LP\times P\cup P\times LP),\\
        q_R\colon& H_*(LP\times LP,LP\times P) \ra H_*(LP\times LP,LP\times P\cup P\times LP)
    \end{align*}
    be induced by the inclusions of pairs.
    For $\nu\in H_1(LP,P)$, following the notation of \cite[Theorem~A]{NaefSafronov}, denote the classes associated with the two relative antidiagonal maps by
    \begin{align*}
        \sum \bar\nu'\times\nu'',\quad \sum \nu'\times\bar\nu''
    \end{align*}
    respectively. Here and below, summation is understood, and the bar indicates the relative factor. For $z\in H_m(LP)$, define
    \begin{align*}
        D_\nu(z)=(-1)^{m-d}\sum q_L(\bar \nu'\times \wedge(\nu''\times z))-\sum q_R(\wedge(z\times \nu')\times \bar \nu''),
    \end{align*}
    where $\wedge$ denotes the string product. We use the same notation \(D_\nu^P\) for the induced map on relative homology. Let
    \begin{align*}
        \nu_1\coloneqq d\log\tau(f)\in H_1(LM_1,M_1)
    \end{align*}
    be the Dennis trace of the Whitehead torsion of $f$, and set
    \begin{align*}
        \nu_4\coloneqq (Lf)_*\nu_1\in H_1(LM_4,M_4).
    \end{align*}
    The transformation formula of Naef and Safronov gives
    \begin{align*}
        \vee((Lf)_*z)=(Lf\times Lf)_*(\vee(z))+D_{\nu_4}((Lf)_*z)
    \end{align*}
    By the naturality of the antidiagonal maps and the string product,
    \begin{align*}
        D_{\nu_4}((Lf)_*z)=(Lf\times Lf)_*D_{\nu_1}(z).
    \end{align*}
    This formula descends to relative homology. Let $Lf_{S^1}\colon L{M_1}_{S^1}\ra L{M_4}_{S^1}$ be the map induced by $Lf$. By the naturality of the transfer and projection maps with respect to $f$, we obtain the transformation formula for the string cobracket
    \begin{align*}
        \vee_{S^1}({Lf_{S^1}}_*x)={(Lf_{S^1}\times Lf_{S^1})}_*(\vee_{S^1}(x))+(\pi\times\pi)_*(Lf\times Lf)_*D_{\nu_1}(\tau_*x).
    \end{align*}
    We now compute 
    \begin{align*}
        \mr{pr}_1(\pi\times\pi)_*D_{\nu_4}((Lf)_*j_*\tau_*z_l), \quad z_l\in H_2^{S^1}(L_l(L(9;1))).
    \end{align*}
    Since all relevant homomorphisms are $\Z$-linear, we may reduce the computation modulo $3$. Since
    \begin{align*}
        \tau(f)&=\frac{(t^{16}-1)(t^4-1)}{(t-1)(t-1)}=\frac{(t^7-1)(t^4-1)}{(t-1)(t-1)} \\
        &=(1+t+\cdots +t^6)(1+t+t^2+t^3)\\
        \tau(f)^{-1}&=\frac{t-1}{t^{16}-1}\frac{t-1}{t^4-1}\\
        &=\frac{t^{9\cdot 7+1}-1}{t^{16}-1}\frac{t^{9\cdot 3+1}-1}{t^4-1}\\
        &=(1+t^{16}+t^{32}+t^{48})(1+t^4+t^8+\cdots +t^{24})\\
        &=(1+t^3+t^5+t^7)(1+t^2+t^3+t^4+t^6+t^7+t^8),
    \end{align*}
    we obtain
    \begin{align*}
        \nu_1&=(2t+2t^3+t^4+t^6)dt\in (\Z/3\Z)[t]/(t^9-1)dt.
    \end{align*}
     After homogenizing, we compute $\mr{pr}_1(\pi\times\pi)_*(Lf\times Lf)_*D_{\nu_1}(j_*\tau_*z_l)$. Here, ``homogenizing'' means rewriting the expression in terms of the homogeneous basis $t^it_2^j\frac{dt}{t}$ (with $i,j$ taken modulo $9$), so that each term lies in a well-defined component $H_1(L_iM)\otimes H_0(L_jM)$, and the projection $\pi _*$ can be applied. (We use the results of \cite[Section 2.4]{Naef24}.)
    \begin{align*}
        &\mr{pr}_1(\pi\times\pi)_*(Lf\times Lf)_*D_{\nu_1}(j_*\tau_*z_l)\\
        &=(\pi\times \pi)_* (Lf\times Lf)_*((t^{l}-t_2^{l})(2tt_2^7+2t^3t_2^5+t^4t_2^4+t^6t_2^2)dt)\\
        &=(\pi\times \pi)_*\left((t^{4l}-t_2^{4l})(2t^4t_2^{28}+2t^{12}t_2^{20}+t^{16}t_2^{16}+t^{24}t_2^8)4t^4\frac{dt}{t}\right).
    \end{align*}
    For example, when $l=1$, we obtain
    \begin{align*}
        \mr{pr}_1(\pi\times\pi)_*(Lf\times Lf)_*D_{\nu_1}(j_*\tau_*z_1)&=2\alpha_3\otimes \beta_1+\alpha_6\otimes \beta_7\equiv2\alpha_6\otimes \beta_7 \mod{W_4}.
    \end{align*}
    Since $f$ sends the component $L_1M$ to $L_4N$, this agrees with the $l=4$ entry of Table \ref{string_cobracket_4}, whereas the $l=1$ entry of Table \ref{string_cobracket_1} vanishes modulo $W_1$. The same calculation for $l=1,\ldots,8$, after relabeling the components by $l\mapsto4l$, reproduces all the discrepancies between Tables \ref{string_cobracket_1} and \ref{string_cobracket_4}.    \begin{corollary}
        The string cobracket is invariant under orientation-preserving simple homotopy equivalences.
    \end{corollary}
    \begin{proof}
        For a simple homotopy equivalence, the Whitehead torsion vanishes. Hence, the correction term in the transformation formula is zero. The result now follows from the naturality of $\tau_*$ and $\pi_*$.
    \end{proof}

    \bibliographystyle{amsalpha}
    \bibliography{Isana_Sumoto_1_rev2}
\end{document}